\documentclass[11pt]{amsart} 
\usepackage[foot]{amsaddr}
\usepackage{amsmath,amssymb,bm,color,float} 
\usepackage{graphicx}
\usepackage{mathrsfs}
\usepackage{dsfont}
\usepackage{hyperref}

 \def\ad#1{\begin{aligned}#1\end{aligned}}  \def\b#1{{\bf #1}} 
\def\a#1{\begin{align*}#1\end{align*}} 
\def\an#1{\begin{align}#1\end{align}}
 \def\t#1{\hbox{\rm{#1}}}

\def\p#1{\begin{pmatrix}#1\end{pmatrix}} 
  \numberwithin{equation}{section}

\def\3bar{{|\hspace{-.02in}|\hspace{-.02in}|}}

\newtheorem{theorem}{Theorem}[section]
\newtheorem{lemma}[theorem]{Lemma}

\def\ad#1{\begin{aligned}#1\end{aligned}} 
\def\b#1{\mathbf{#1}} 

\def\t#1{\operatorname{#1}}

\numberwithin{equation}{section}

\title[Least-Squares Weak Galerkin]{A Least-Squares Weak Galerkin Finite Element Scheme for Cauchy Problems in Convection--Diffusion}

\author{Chunmei Wang} 
\address{Department of Mathematics, University of Florida, Gainesville, FL 32611, USA.} 
\email{chunmei.wang@ufl.edu} 

\author{Shangyou Zhang}
\address{Department of Mathematical Sciences, University of Delaware, Newark, DE 19716, USA} 
\email{szhang@udel.edu}

\begin{document}
\begin{abstract}
We introduce and rigorously analyze a least-squares weak Galerkin (LS-WG) finite element method for the severely ill-posed Cauchy problem of convection--diffusion equations. The proposed framework utilizes weak derivatives defined on a class of discontinuous weak functions, enabling the natural treatment of complex boundary conditions and internal interfaces. A key advantage of the least-squares formulation is that it transforms the underlying non-self-adjoint operator into a discrete linear system that is inherently symmetric and positive definite (SPD). We demonstrate the geometric flexibility of the method on arbitrary polygonal and polyhedral partitions. Furthermore, we establish the uniqueness of the numerical solution and derive optimal-order error estimates in a carefully defined discrete energy norm. Extensive numerical tests are presented to confirm the theoretical convergence rates and highlight the algorithm's robustness and efficiency compared to standard Galerkin approaches.
\end{abstract}

 
  \maketitle
\section{Introduction} 
In this work, we introduce a least-squares weak Galerkin (LS-WG) finite element method based on completely discontinuous functions to solve Cauchy problems for convection--diffusion equations. To present the proposed method, we consider the following model problem: find a scalar function $u$ such that
\begin{equation}\label{model}
\begin{cases} 
-\epsilon \Delta u + \mathbf{b} \cdot \nabla u = f & \text{in } \Omega, \\
u = g_1, \quad \nabla u \cdot \mathbf{n} = g_2 & \text{on } \Gamma_1,
\end{cases}
\end{equation}
where $\Omega \subset \mathbb{R}^d$ ($d=2, 3$) is an open, bounded, convex polytopal domain. The boundary $\partial\Omega$ is partitioned into two disjoint, relatively open subsets $\Gamma_1$ and $\Gamma_2$, such that $\partial\Omega = \overline{\Gamma}_1 \cup \overline{\Gamma}_2$. Here, $\Gamma_1$ (with strictly positive measure, $|\Gamma_1|>0$) represents the accessible boundary where both Dirichlet and Neumann data are prescribed, while $\Gamma_2$ is the inaccessible boundary where no data is provided. 

The physical parameters consist of the constant diffusion coefficient $\epsilon > 0$ and the convective velocity field $\mathbf{b} \in [L^\infty(\Omega)]^d$. The vector $\mathbf{n}$ denotes the unit outward normal to $\partial\Omega$. The source term $f$ and the boundary data $g_1, g_2$ are assumed to be sufficiently regular to ensure the well-posedness of the corresponding discrete formulation. The objective of this study is to reconstruct the solution $u$ throughout the domain $\Omega$ and its trace on the unknown boundary $\Gamma_2$. Given the inherent ill-posedness of this boundary value problem, we employ a stabilized numerical approach through the LS-WG framework to ensure theoretical accuracy and computational robustness.

The convection--diffusion equation serves as a foundational mathematical framework for modeling the transport of physical quantities---such as heat, mass, and chemical concentrations---within a moving medium. While the Cauchy problem for the Poisson equation is widely utilized in fields like plasma physics, electrocardiography, and non-destructive evaluation, the inclusion of the first-order convection term, $\mathbf{b} \cdot \nabla u$, significantly expands this utility. 

In environmental engineering, for instance, this framework is essential for the ``inverse tracking'' of pollutant sources in groundwater or atmospheric flows based on downstream sensor measurements. Similarly, in aerospace and industrial cooling, it facilitates non-intrusive thermal imaging and the estimation of hidden boundary heat fluxes. Other critical applications span biomedical modeling for drug transport reconstruction in the bloodstream and petroleum engineering for characterizing flow paths in porous media via tracer analysis.

Despite its versatility, the Cauchy problem for the convection--diffusion equation introduces substantial mathematical and numerical complexities. Like its pure elliptic counterpart, it is severely ill-posed; the solution lacks continuous dependence on the given data, meaning that infinitesimal high-frequency noise in the measurements on $\Gamma_1$ can result in exponentially large errors in the interior domain. Furthermore, the convection operator introduces unique structural hurdles---most notably, the underlying differential operator is non-self-adjoint. This lack of symmetry significantly complicates the stability analysis and degrades the efficiency of traditional numerical solvers. 

Moreover, in convection-dominated regimes where the diffusion coefficient $\epsilon$ is sufficiently small, the solution often exhibits sharp boundary or interior layers. In these regions, standard finite element methods frequently succumb to non-physical, spurious oscillations. This difficulty is exacerbated by the directional nature of the flow: reconstructing a solution ``upstream'' (against the velocity field $\mathbf{b}$) is inherently more unstable than downstream reconstruction, as the physical convective process naturally dissipates source information as it moves through the domain.

The existing literature reflects this intersection of severe ill-posedness and complex numerical instability. The directional bias inherent in the first-order term implies that upstream reconstruction is exponentially more sensitive to measurement noise than downstream efforts \cite{Ranjbar}. Historically, these challenges have been addressed through various regularization strategies, including Tikhonov-type methods, iterative regularization, and singular value decomposition (SVD) to quantify the decay rate of the operator’s spectrum \cite{He, Ranjbar}. In convection-dominated transport, standard Galerkin methods often fail due to the emergence of sharp layers, prompting the development of stabilized schemes such as streamline-diffusion or super-localized orthogonal decomposition \cite{Bonizzoni}.  

The Weak Galerkin (WG) finite element method, originally introduced in \cite{ellip_JCAM2013} and extensively expanded in subsequent works \cite{wg1, wg2, wg3, wg4, wg5, wg6, wg7, wg8, wg9, wg10, wg11, wg12, wg13, wg14, wg15, wg16, wg17, wg18, wg19, wg20, wg21, itera, wz2023, wy3655}, represents a significant departure from traditional continuous finite element methods. By utilizing weak derivatives and enforcing weak continuity across element interfaces via specifically designed stabilizers, the WG framework provides a highly flexible platform that naturally accommodates general polygonal and polyhedral meshes. 

A notable evolution of this framework is the Primal-Dual Weak Galerkin (PDWG) method \cite{pdwg1, pdwg2, pdwg3, pdwg4, pdwg5, pdwg6, pdwg7, pdwg8, pdwg9, pdwg10, pdwg11, pdwg12, pdwg13, pdwg14, pdwg15}, which formulates numerical approximations as constrained minimization problems. In PDWG, the governing equations are enforced as weak constraints through Lagrange multipliers. This approach yields Euler--Lagrange systems that involve both primal and dual variables, offering favorable stability and symmetry properties, even for non-self-adjoint problems such as linear transport equations \cite{wwhyperbolic, pdwg1, pdwg6, pdwg12}.

Building upon these developments, the objective of this paper is to introduce a novel \emph{Least-Squares Weak Galerkin (LS-WG) finite element method} specifically designed for the Cauchy problem of convection--diffusion equations \eqref{model}. Unlike traditional WG formulations that lead to non-symmetric systems for convection problems, or PDWG methods that significantly increase the number of global unknowns through dual variables, the proposed LS-WG method transforms the governing equations into a least-squares functional minimized in a weak sense. This formulation offers several distinct mathematical and computational advantages:
\begin{itemize}
    \item \textbf{Symmetry and Positivity:} The resulting discretized bilinear form is inherently symmetric and positive definite, leading to an SPD linear system that allows for the use of highly efficient iterative solvers, such as the Conjugate Gradient (CG) method.
    \item \textbf{Robustness:} The least-squares framework provides a natural, built-in stabilization mechanism for convection-dominated regimes, effectively mitigating the spurious oscillations often encountered in standard Galerkin formulations.
    \item \textbf{Mesh Flexibility:} The method retains the hallmark WG advantage of being applicable to arbitrary polytopal meshes with pendant nodes, without requiring matching grids or highly specialized continuous basis functions.
\end{itemize}

In this work, we provide a rigorous theoretical foundation for the proposed LS-WG scheme. We establish the uniqueness of the discrete solution without requiring standard inf-sup conditions. Furthermore, we derive optimal-order error estimates in a carefully defined discrete energy norm. Extensive numerical experiments are presented to validate these theoretical findings and to demonstrate the method's accuracy and robustness in reconstructing solutions to severely ill-posed Cauchy problems. 

The remainder of this paper is organized as follows. In Section~2, we provide a brief mathematical review of the weak gradient, the weak Laplacian, and their corresponding discrete counterparts. Section~3 introduces the detailed LS-WG formulation tailored for the convection--diffusion Cauchy problem. The theoretical foundation of the method is addressed in Section~4, where we establish the   uniqueness of the numerical solution. In Section~5, we derive the error equations for the LS-WG finite element approximation, which serves as the critical basis for the subsequent convergence analysis. Section~6 is devoted to the rigorous proof of optimal-order error estimates in a discrete Sobolev norm. Finally, in Section~7, we present a series of numerical experiments to demonstrate the stability, accuracy, and efficiency of the proposed LS-WG framework.

Throughout this work, we adopt standard notation for Sobolev spaces and their associated norms. For any open, bounded domain $D \subset \mathbb{R}^d$ with a Lipschitz continuous boundary, $\|\cdot\|_{s,D}$ and $|\cdot|_{s,D}$ denote the norm and seminorm of the Sobolev space $H^s(D)$ for $s \ge 0$, respectively. The corresponding inner product is denoted by $(\cdot, \cdot)_{s,D}$. In the special case where $s=0$, the space $H^0(D)$ coincides with $L^2(D)$, with the norm and inner product denoted by $\|\cdot\|_D$ and $(\cdot, \cdot)_D$, respectively. For simplicity, the subscript $D$ is omitted when $D = \Omega$ or when the domain of integration is clear from the context.

\section{Discrete Weak Gradient and Discrete Weak Laplacian}

In this section, we review the fundamental definitions of the weak gradient and weak Laplacian operators, along with their discrete counterparts, which are essential for the construction of the  WG finite element framework.

Let $T$ be a polygonal domain in $\mathbb{R}^2$ or a polyhedral domain in $\mathbb{R}^3$ with boundary $\partial T$. A \emph{weak function} on $T$ is defined as an ordered triplet $v=\{v_0, v_b, \mathbf{v}_g\}$, where $v_0 \in L^2(T)$ represents the value of $v$ in the interior of $T$, while $v_b \in L^2(\partial T)$ and $\mathbf{v}_g \in [L^2(\partial T)]^d$ represent the values of $v$ and its gradient $\nabla v$ on the boundary $\partial T$, respectively. Crucially, $v_b$ and $\mathbf{v}_g$ are defined independently and are not required to be the traces of $v_0$ and $\nabla v_0$ on $\partial T$, although such a choice is admissible.

We denote the space of all weak functions on $T$ by $\mathcal{W}(T)$, defined as:
\an{ \label{eq:weakspace} \ad{
\mathcal{W}(T) = \Big\{ v = \{v_0, v_b, \mathbf{v}_g\} : v_0 \in L^2(T),& \ v_b \in L^2(\partial T), \\
           &\ \mathbf{v}_g \in [L^2(\partial T)]^d \Big\}. }
 }

\subsection{Weak and Discrete Weak Laplacian}
The \emph{weak Laplacian} of a function $v \in \mathcal{W}(T)$, denoted by $\Delta_w v$, is defined as a linear functional in the dual space of $H^2(T)$ such that
\begin{equation}\label{eq:weak_laplacian_def}
(\Delta_w v, w)_T = (v_0, \Delta w)_T - \langle v_b, \nabla w \cdot \mathbf{n} \rangle_{\partial T} + \langle \mathbf{v}_g \cdot \mathbf{n}, w \rangle_{\partial T}, \quad \forall w \in H^2(T),
\end{equation}
where $\mathbf{n}$ denotes the outward unit normal vector on $\partial T$.

Let $P_r(T)$ be the space of polynomials of degree at most $r$ on $T$. The \emph{discrete weak Laplacian} of $v \in \mathcal{W}(T)$, denoted by $\Delta_{w,r,T} v$, is the unique polynomial in $P_r(T)$ satisfying
\begin{equation}\label{dislap}
(\Delta_{w,r,T} v, w)_T = (v_0, \Delta w)_T - \langle v_b, \nabla w \cdot \mathbf{n} \rangle_{\partial T} + \langle \mathbf{v}_g \cdot \mathbf{n}, w \rangle_{\partial T}, \quad \forall w \in P_r(T).
\end{equation}
If the interior component $v_0$ possesses $H^2$ regularity (i.e., $v_0 \in H^2(T)$), the discrete weak Laplacian can be equivalently expressed through integration by parts as:
\an{\label{dislap2} \ad{
(\Delta_{w,r,T} v, w)_T = &  (\Delta v_0, w)_T + \langle v_0 - v_b, \nabla w \cdot \mathbf{n} \rangle_{\partial T} \\
                  & \ + \langle (\mathbf{v}_g - \nabla v_0) \cdot \mathbf{n}, w \rangle_{\partial T}, \forall w \in P_r(T). }
 }

\subsection{Weak and Discrete Weak Gradient}
The \emph{weak gradient} of a function $v \in \mathcal{W}(T)$, denoted by $\nabla_w v$, is defined as a linear functional in the dual space of $[H^1(T)]^d$ such that
\begin{equation}\label{eq:weak_gradient_def}
(\nabla_w v, \bm{\psi})_T = -(v_0, \nabla \cdot \bm{\psi})_T + \langle v_b, \bm{\psi} \cdot \mathbf{n} \rangle_{\partial T}, \quad \forall \bm{\psi} \in [H^1(T)]^d.
\end{equation}

The \emph{discrete weak gradient} $\nabla_{w,r,T} v$ is the unique vector-valued polynomial in $[P_r(T)]^d$ satisfying
\begin{equation}\label{disgradient}
(\nabla_{w,r,T} v, \bm{\psi})_T = -(v_0, \nabla \cdot \bm{\psi})_T + \langle v_b, \bm{\psi} \cdot \mathbf{n} \rangle_{\partial T}, \quad \forall \bm{\psi} \in [P_r(T)]^d.
\end{equation}
Furthermore, if $v_0 \in H^1(T)$, the discrete weak gradient admits the following equivalent form via integration by parts:
\begin{equation}\label{disgradient_star}
(\nabla_{w,r,T} v, \bm{\psi})_T = (\nabla v_0, \bm{\psi})_T - \langle v_0 - v_b, \bm{\psi} \cdot \mathbf{n} \rangle_{\partial T}, \quad \forall \bm{\psi} \in [P_r(T)]^d.
\end{equation}
\section{Least-Squares Weak Galerkin Algorithm}\label{Section:WGFEM}

In this section, we describe the least-squares weak Galerkin   finite element discretization for the convection--diffusion Cauchy problem \eqref{model}. 

Let $\mathcal{T}_h$ be a shape-regular partition of the domain $\Omega$ into polygonal (2D) or polyhedral (3D) elements as defined in \cite{wy3655}. We denote by $\mathcal{E}_h$ the set of all edges (or faces in 3D) in $\mathcal{T}_h$, and let $\mathcal{E}_h^0 = \mathcal{E}_h \setminus \partial\Omega$ be the set of all interior edges (or faces in 3D). For each element $T \in \mathcal{T}_h$, $h_T$ denotes the diameter of $T$, and the mesh size is given by $h = \max_{T \in \mathcal{T}_h} h_T$.

For a given integer $k \ge 1$, we define the local weak finite element space $W_k(T)$ as:
\an{\label{W-k} \ad{
W_k(T) = \bigl\{ v = \{v_0, v_b, \mathbf{v}_g\} : & v_0 \in P_k(T); \ v_b \in P_k(e), \\
    & \  \mathbf{v}_g \in [P_{k-1}(e)]^d, \ e \subset \partial T \bigr\}. } }
The global weak finite element space $W_h$ is constructed by patching the local spaces $W_k(T)$ across interior edges such that $v_b$  is single-valued on each $e \in \mathcal{E}_h$. We define the subspace $W_h^0 \subset W_h$ to be the set of weak functions with vanishing Cauchy data on the boundary $\Gamma_1$:
\[
W_h^0 = \bigl\{ v = \{v_0, v_b, \mathbf{v}_g\} \in W_h : v_b|_e = 0, \ \mathbf{v}_g \cdot \mathbf{n}|_e = 0, \ e \subset \Gamma_1 \bigr\}.
\]

For any $v \in W_h$, let $\Delta_w v$ and $\nabla_w v$ be the discrete weak Laplacian and discrete weak gradient, respectively, computed element-wise as:
\[
(\Delta_w v)|_T = \Delta_{w,k-1,T}(v|_T), \quad (\nabla_w v)|_T = \nabla_{w,k-1,T}(v|_T), \quad \forall T \in \mathcal{T}_h.
\]

To enforce the connection between the interior and boundary components of the weak functions, we introduce the following stabilizer $s(\cdot, \cdot)$:
\an{ \label{stabilizer} \ad{
s(u, v) = \sum_{T \in \mathcal{T}_h}   & h_T^{-3} \langle u_0 - u_b, v_0 - v_b \rangle_{\partial T} + \\
         & \ h_T^{-1} \langle (\nabla u_0 - \mathbf{u}_g) \cdot \mathbf{n}, (\nabla v_0 - \mathbf{v}_g) \cdot \mathbf{n} \rangle_{\partial T},  \forall u, v \in W_h. 
 } }
The least-squares bilinear form $a(\cdot, \cdot)$ for the LS-WG method is then defined as:
\begin{equation}\label{bilinear_a}
a(u, v) = \sum_{T \in \mathcal{T}_h} \bigl( -\epsilon \Delta_w u + \mathbf{b} \cdot \nabla_w u, -\epsilon \Delta_w v + \mathbf{b} \cdot \nabla_w v \bigr)_T + s(u, v), \quad \forall u, v \in W_h.
\end{equation}

Next, we define the necessary $L^2$ projection operators. For each $T \in \mathcal{T}_h$ and its edges $e \subset \partial T$, let:
\begin{itemize}
    \item $Q_0$: $L^2(T) \to P_k(T)$ be the local projection onto the interior polynomial space;
    \item $Q_b$: $L^2(e) \to P_k(e)$ be the local projection onto the boundary polynomial space;
    \item $Q_n$: $L^2(e) \to P_{k-1}(e)$ be the local projection for the normal component of the gradient;
    \item $\mathbf{Q}_g$: $[L^2(e)]^d \to [P_{k-1}(e)]^d$ be the local vector projection for the gradient component.
\end{itemize}
For the exact solution $u$ of \eqref{model}, its projection into the global space $W_h$ is denoted by $Q_h u = \{Q_0 u, Q_b u, \mathbf{Q}_g (\nabla u)\}$.

The LS-WG finite element scheme for the Cauchy problem \eqref{model} is formulated as follows:

\noindent \textbf{Algorithm 3.1 (LS-WG Scheme).} 
Find $u_h = \{u_0, u_b, \mathbf{u}_g\} \in W_h$ satisfying the boundary conditions $u_b = Q_b g_1$ and $\mathbf{u}_g \cdot \mathbf{n} = Q_n g_2$ on $\Gamma_1$, such that
\begin{equation}\label{al-general}
a(u_h, v_h) = \sum_{T \in \mathcal{T}_h} (f, -\epsilon \Delta_w v_h + \mathbf{b} \cdot \nabla_w v_h)_T, \quad \forall v_h \in W_h^0.
\end{equation}
 
This least-squares formulation naturally transforms the non-symmetric convection--diffusion operator into a symmetric and positive-definite (SPD) system. The inclusion of the stabilizer $s(u, v)$ ensures that the method remains well-posed and that the interior and boundary components converge at the optimal rate.

\section{Solution  Uniqueness}\label{Section:EU}

In this section, we establish the  uniqueness of the solution to the LS-WG scheme \eqref{al-general}. We begin by identifying the commutative properties of the $L^2$ projection operators, which are essential for the subsequent analysis.

Let $\mathcal{Q}_h^{k-1}$ denote the locally defined $L^2$ projections onto the polynomial space $P_{k-1}(T)$ (for scalars) or $[P_{k-1}(T)]^d$ (for vectors) for each element $T \in \mathcal{T}_h$.

\begin{lemma}[Commutative Properties]\label{Lemma:commute}
The $L^2$ projection operators $Q_h$  and $\mathcal{Q}_h^{k-1}$ satisfy the following commutative properties:
\begin{equation}\label{eq:commute1}
\nabla_w(Q_h w) = \mathcal{Q}_h^{k-1}(\nabla w), \quad \forall w \in H^1(\Omega),
\end{equation}
\begin{equation}\label{eq:commute2}
\Delta_w(Q_h w) = \mathcal{Q}_h^{k-1}(\Delta w), \quad \forall w \in H^2(\Omega).
\end{equation}
\end{lemma}

\begin{proof}
For any $\bm{\psi} \in [P_{k-1}(T)]^d$, the definition of the discrete weak gradient \eqref{disgradient} implies:
\[
\begin{aligned}
(\nabla_w Q_h w, \bm{\psi})_T &= -(Q_0 w, \nabla \cdot \bm{\psi})_T + \langle Q_b w, \bm{\psi} \cdot \mathbf{n} \rangle_{\partial T} \\
&= -(w, \nabla \cdot \bm{\psi})_T + \langle w, \bm{\psi} \cdot \mathbf{n} \rangle_{\partial T} \\
&= (\nabla w, \bm{\psi})_T = (\mathcal{Q}_h^{k-1} \nabla w, \bm{\psi})_T,
\end{aligned}
\]
where we used the fact that $\nabla \cdot \bm{\psi} \in P_{k-2}(T) \subset P_k(T)$ and $\bm{\psi} \cdot \mathbf{n} \in P_{k-1}(e) \subset P_k(e)$. This confirms \eqref{eq:commute1}.

Similarly, for any $q \in P_{k-1}(T)$, the definition of the discrete weak Laplacian \eqref{dislap} gives:
\[
\begin{aligned}
(\Delta_w Q_h w, q)_T &= (Q_0 w, \Delta q)_T - \langle Q_b w, \nabla q \cdot \mathbf{n} \rangle_{\partial T} + \langle \mathbf{Q}_g(\nabla w) \cdot \mathbf{n}, q \rangle_{\partial T} \\
&= (w, \Delta q)_T - \langle w, \nabla q \cdot \mathbf{n} \rangle_{\partial T} + \langle \nabla w \cdot \mathbf{n}, q \rangle_{\partial T} \\
&= (\Delta w, q)_T = (\mathcal{Q}_h^{k-1} \Delta w, q)_T,
\end{aligned}
\]
where we used $\Delta q \in P_{k-3}(T) \subset P_k(T)$, $\nabla q \cdot \mathbf{n} \in P_{k-2}(e) \subset P_k(e)$, and $q \in P_{k-1}(e)$. This proves \eqref{eq:commute2}.
\end{proof}

\begin{lemma}[Uniqueness]\label{thmunique1}
Assume that the continuous Cauchy problem \eqref{model} admits a unique solution. Then the LS-WG finite element scheme \eqref{al-general} possesses a unique solution $u_h \in W_h$.
\end{lemma}

\begin{proof}  It suffices to show that the homogeneous problem ($f=0, g_1=0, g_2=0$) admits only the trivial solution $u_h = 0$. Let $u_h \in W_h^0$ satisfy
\[
a(u_h, v_h) = 0, \quad \forall v_h \in W_h^0.
\]
By choosing $v_h = u_h$, we obtain $a(u_h, u_h) = 0$, which by the definition of the bilinear form \eqref{bilinear_a} implies:
\begin{enumerate}
    \item $s(u_h, u_h) = 0 \implies u_0 = u_b$ and $(\nabla u_0 - \mathbf{u}_g) \cdot \mathbf{n} = 0$ on $\partial T$ for all $T \in \mathcal{T}_h$.
    \item $-\epsilon \Delta_w u_h + \mathbf{b} \cdot \nabla_w u_h = 0$ on each $T \in \mathcal{T}_h$.
\end{enumerate}

From the property $u_0 = u_b$ on $\partial T$ and the fact that $u_b$ is single-valued on interior edges $\mathcal{E}_h^0$, we conclude that $u_0$ is continuous across all element interfaces, i.e., $u_0 \in H^1(\Omega)$. Furthermore, since $(\nabla u_0 - \mathbf{u}_g) \cdot \mathbf{n} = 0$ and $\mathbf{u}_g \cdot \mathbf{n}$ is single-valued on interior edges, the normal component of the gradient $\nabla u_0 \cdot \mathbf{n}$ is also continuous across interfaces. This implies $u_0 \in H^2(\Omega)$.

Using the equivalent definitions \eqref{dislap2} and \eqref{disgradient_star} with $u_0 = u_b$ and $\nabla u_0 \cdot \mathbf{n} = \mathbf{u}_g \cdot \mathbf{n}$, we find that on each element $T$:
\[
\Delta_w u_h =  \Delta u_0  \quad \text{and} \quad \nabla_w u_h =  \nabla u_0.
\]
Substituting these into the residual equation yields:
\[
-\epsilon \Delta u_0 + \mathbf{b} \cdot \nabla u_0 = 0 \quad \text{in } \Omega.
\]
Finally, the boundary conditions for $u_h \in W_h^0$ imply $u_b = 0$ and $\mathbf{u}_g \cdot \mathbf{n} = 0$ on $\Gamma_1$, which translates to $u_0 = 0$ and $\nabla u_0 \cdot \mathbf{n} = 0$ on $\Gamma_1$. By the uniqueness assumption of the continuous Cauchy problem, $u_0 \equiv 0$ in $\Omega$, which implies $u_b = 0$ and $\mathbf{u}_g = 0$. Thus, $u_h \equiv 0$.
\end{proof}

We define the energy norm on the finite element space $W_h^0$ by
\begin{equation}
\3bar v\3bar:= \sqrt{a(v, v)}.
\end{equation}
Following the logic of Lemma \ref{thmunique1}, it is evident that $\3bar\cdot\3bar$ satisfies the properties of a norm on $W_h^0$, providing a robust framework for the subsequent error analysis.
 
 \section{Error Equations}\label{Section:ErrorEquation}

In this section, we derive the error equation that governs the relationship between the exact solution $u$ and its LS-WG approximation $u_h$. Let $u$ be the exact solution of \eqref{model}, and let $u_h \in W_h$ be the solution to the discrete problem \eqref{al-general}. We define the error function as
\begin{equation}\label{error}
e_h := u_h - Q_h u = \{u_0 - Q_0 u, \ u_b - Q_b u, \ \mathbf{u}_g - \mathbf{Q}_g \nabla u\}.
\end{equation}
For the simplicity of the error analysis, we assume that the diffusion coefficient $\epsilon > 0$ and the convective velocity field $\mathbf{b}$ are piecewise constant with respect to the partition $\mathcal{T}_h$.   This assumption simplifies the presentation, and the results can be extended to piecewise smooth coefficients by considering additional higher-order consistency terms.

\begin{lemma}[Error Equation]\label{errorequa}
For any test function $v_h \in W_h^0$, the error function $e_h$ satisfies the following identity:
\begin{equation}\label{erroreqn}
a(e_h, v_h) = -s(Q_h u, v_h).
\end{equation}
\end{lemma}

\begin{proof}
Let $\mathcal{L}$ denote the continuous operator $\mathcal{L} u := -\epsilon \Delta u + \mathbf{b} \cdot \nabla u = f$. Testing this continuous equation with the discrete operator $(\mathcal{L}_w v_h) |_T := -\epsilon \Delta_w v_h + \mathbf{b} \cdot \nabla_w v_h$ on each element $T \in \mathcal{T}_h$, we have
\begin{equation}\label{eq:err_proof1}
\sum_{T \in \mathcal{T}_h} (-\epsilon \Delta u + \mathbf{b} \cdot \nabla u, -\epsilon \Delta_w v_h + \mathbf{b} \cdot \nabla_w v_h)_T = \sum_{T \in \mathcal{T}_h}(f, -\epsilon \Delta_w v_h + \mathbf{b} \cdot \nabla_w v_h)_T.
\end{equation}
Note that $(-\epsilon \Delta_w v_h + \mathbf{b} \cdot \nabla_w v_h) \in P_{k-1}(T)$ on each element $T$. By applying the commutative properties \eqref{eq:commute1}--\eqref{eq:commute2} established in Lemma \ref{Lemma:commute}, we have:
\[
\mathcal{Q}_h^{k-1}(\Delta u) = \Delta_w Q_h u, \quad \mathcal{Q}_h^{k-1}(\nabla u) = \nabla_w Q_h u.
\]
Substituting these identities into \eqref{eq:err_proof1} yields
\an{\label{eq:err_proof2} \ad{ & \quad \ 
\sum_{T \in \mathcal{T}_h} (-\epsilon \Delta_w Q_h u + \mathbf{b} \cdot \nabla_w Q_h u, -\epsilon \Delta_w v_h + \mathbf{b} \cdot \nabla_w v_h)_T
\\ & =\sum_{T \in \mathcal{T}_h} (f, -\epsilon \Delta_w v_h + \mathbf{b} \cdot \nabla_w v_h)_T. }}
Recalling the definition of the least-squares bilinear form $a(\cdot, \cdot)$ in \eqref{bilinear_a}, the left-hand side of \eqref{eq:err_proof2} can be rewritten as $a(Q_h u, v_h) - s(Q_h u, v_h)$. Thus,
\[
a(Q_h u, v_h) - s(Q_h u, v_h) =\sum_{T \in \mathcal{T}_h} (f, -\epsilon \Delta_w v_h + \mathbf{b} \cdot \nabla_w v_h)_T.
\]
Subtracting this identity from the discrete LS-WG scheme \eqref{al-general}, we obtain
\[
a(u_h - Q_h u, v_h) + s(Q_h u, v_h) = 0,
\]
which simplifies to the desired error equation $a(e_h, v_h) = -s(Q_h u, v_h)$.
\end{proof}

\section{Error Estimates}\label{Section:ErrorEstimates}

In this section, we establish optimal-order error estimates for the LS-WG approximation in the energy norm. Throughout the analysis, we denote by $C$ a generic positive constant independent of the mesh parameter $h$.

\begin{lemma}[Approximation Properties]\label{lem:approx}
Let $\mathcal{T}_h$ be a shape-regular finite element partition of $\Omega$. For any $u \in H^{k+1}(\Omega)$, the following approximation estimates hold:
\begin{equation}\label{error_L2}
\sum_{T \in \mathcal{T}_h} \|u - Q_0 u\|_T^2 \le C h^{2k+2} \|u\|_{k+1}^2,
\end{equation}
\begin{equation}\label{error_H1}
\sum_{T \in \mathcal{T}_h} |u - Q_0 u|_{s,T}^2 \le C h^{2(k+1-s)} \|u\|_{k+1}^2, \quad s=1, 2.
\end{equation}
Furthermore, for any $\phi \in H^1(T)$, the following trace inequality is valid:
\begin{align}
\|\phi\|_{\partial T}^2 &\le C (h_T^{-1} \|\phi\|_T^2 + h_T \|\nabla \phi\|_T^2). \label{trace_H1} 
\end{align}
\end{lemma}

\begin{theorem}[Convergence in Energy Norm]\label{thm:convergence}
Let $u \in H^{k+1}(\Omega)$ be the exact solution of the convection--diffusion Cauchy problem \eqref{model}, and let $u_h \in W_h$ be the LS-WG solution defined by \eqref{al-general}. Then, there exists a constant $C > 0$ such that
\begin{equation}\label{main_estimate}
\3bar  u_h - Q_h u \3bar  \le C h^{k-1} \|u\|_{k+1}.
\end{equation}
\end{theorem}

\begin{proof}
By setting the test function $v_h = e_h$ in the error equation \eqref{erroreqn}, we have:
\[
\3bar  e_h \3bar ^2 = a(e_h, e_h) = -s(Q_h u, e_h).
\]
Applying the Cauchy--Schwarz inequality to the stabilization term $s(\cdot, \cdot)$, we obtain:
\[
\3bar  e_h \3bar ^2 \le |s(Q_h u, e_h)| \le \sqrt{s(Q_h u, Q_h u)} \sqrt{s(e_h, e_h)} \le \sqrt{s(Q_h u, Q_h u)} \3bar  e_h \3bar ,
\]
which implies $\3bar  e_h \3bar  \le \sqrt{s(Q_h u, Q_h u)}$. We now bound the two components of the stabilizer $s(Q_h u, Q_h u)$.

\noindent \textbf{Step 1:} Bound for $\sum_{T \in \mathcal{T}_h} h_T^{-3} \|Q_0 u - Q_b u\|_{\partial T}^2$. Using the fact that $Q_b u = Q_b (u|_{\partial T})$ and applying the trace inequality \eqref{trace_H1} and estimates \eqref{error_L2}-\eqref{error_H1}:
\[
\begin{aligned}
&\sum_{T \in \mathcal{T}_h} h_T^{-3} \|Q_0 u - Q_b u\|_{\partial T}^2 \\&\le \sum_{T \in \mathcal{T}_h} h_T^{-3} \|Q_0 u - u\|_{\partial T}^2 \\
&\le C \sum_{T \in \mathcal{T}_h} h_T^{-3} \left( h_T^{-1} \|Q_0 u - u\|_T^2 + h_T |Q_0 u - u|_{1,T}^2 \right) \\
&\le C \sum_{T \in \mathcal{T}_h} (h_T^{-4} h_T^{2k+2} + h_T^{-2} h_T^{2k}) \|u\|_{k+1,T}^2 \le C h^{2k-2} \|u\|_{k+1}^2.
\end{aligned}
\]

\noindent \textbf{Step 2:} Bound for $\sum_{T \in \mathcal{T}_h} h_T^{-1} \|(\nabla Q_0 u - \mathbf{Q}_g \nabla u) \cdot \mathbf{n}\|_{\partial T}^2$. Applying the trace inequality \eqref{trace_H1}   and  estimates \eqref{error_L2}-\eqref{error_H1}:
\[
\begin{aligned}
&\sum_{T \in \mathcal{T}_h} h_T^{-1} \|(\nabla Q_0 u - \mathbf{Q}_g \nabla u) \cdot \mathbf{n}\|_{\partial T}^2 \\&\le \sum_{T \in \mathcal{T}_h} h_T^{-1} \|\nabla Q_0 u - \nabla u\|_{\partial T}^2 \\
&\le C \sum_{T \in \mathcal{T}_h} h_T^{-1} \left( h_T^{-1} |\nabla Q_0 u - \nabla u|_{0,T}^2 + h_T |\nabla Q_0 u - \nabla u|_{1,T}^2 \right) \\
&\le C \sum_{T \in \mathcal{T}_h} (h_T^{-2} h_T^{2k} + h_T^{2k-2}) \|u\|_{k+1,T}^2 \le C h^{2k-2} \|u\|_{k+1}^2.
\end{aligned}
\]
Combining the estimates from Step 1 and Step 2, we conclude $\3bar e_h \3bar ^2 \le C h^{2k-2} \|u\|_{k+1}^2$, which completes the proof.
\end{proof}

\section{Numerical Experiments}

In the first numerical test,  we solve the Cauchy problem \eqref{model}
   on the unit square domain $\Omega=(0,1)\times(0,1)$, where 
\a{ \Gamma_1=\{0\}\times (0,1)\cup (0,1)\times\{0\}, \ \ \epsilon=10^{-2} \ \t{or} \ 10^{-7 } , \ \ \b b=\p{1\\1}, }
and  $(f, g_1, g_2)$ are chosen so that the exact solution is
\an{\label{s2}
   u =-(2x^3 + y + 1)^2.  }  
 We compute the solution \eqref{s2} on the triangular grids shown in Figure \ref{f-2}, and on the non-convex polygonal 
  grids shown in Figure \ref{f-5}, by 
  the weak Galerkin $P_k$-$P_k$-$P_{k-1}^2$/$P_k^{2\times 2}$ finite elements, $k=2,3$ and $4$.
The results are listed in Tables \ref{t1}-\ref{t3}, 
    where we can see that the optimal orders of convergence 
  are achieved roughly.  
  In these tables, $G_i$ denotes the $i$-th grid.  For example, $G_1$ in Table \ref{t1} is shown in
    Figure \ref{f-2} or Figure \ref{f-5} .
Also in these tables, the $H^2$-like norm
\a{ \3bar u \3bar^2_1 = a(u,u) \quad \t{with } \ \epsilon=1, }
where $a(u,u)$ is defined in \eqref{bilinear_a}.
    
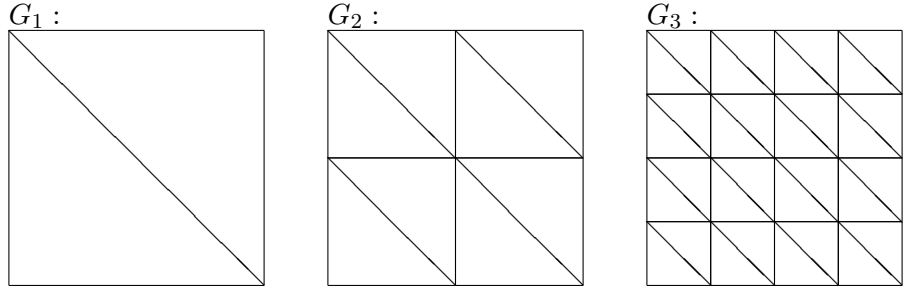
\begin{figure}[H]
\begin{center}\setlength\unitlength{2.4pt}\centering 
 \begin{picture}(140,45)(0,0) \put(0,41){$G_1:$}  \put(50,41){$G_2:$} \put(100,41){$G_3:$} 
  
\def\sq{\begin{picture}(40,40)(0,0) \put(0,40){\line(1,-1){40}}
  \multiput(0,0)(40,0){2}{\line(0,1){40}}\multiput(0,0)(0,40){2}{\line(1,0){40}} \end{picture} }
  
\put(0,0){\begin{picture}(40,40)(0,0)
  \multiput(0,0)(0,40){1}{\multiput(0,0)(40,0){1}{\sq}} 
  \end{picture} }
  
\put(50,0){\setlength\unitlength{1.2pt}\begin{picture}(40,40)(0,0)
  \multiput(0,0)(0,40){2}{\multiput(0,0)(40,0){2}{\sq}} 
  \end{picture} } 
\put(100,0){\setlength\unitlength{0.6pt}\begin{picture}(40,40)(0,0)
  \multiput(0,0)(0,40){4}{\multiput(0,0)(40,0){4}{\sq}} 
  \end{picture} } 
\end{picture}\end{center}
\caption{The triangular  grids used in Tables \ref{t1}--\ref{t6}. }
\label{f-2}
\end{figure}

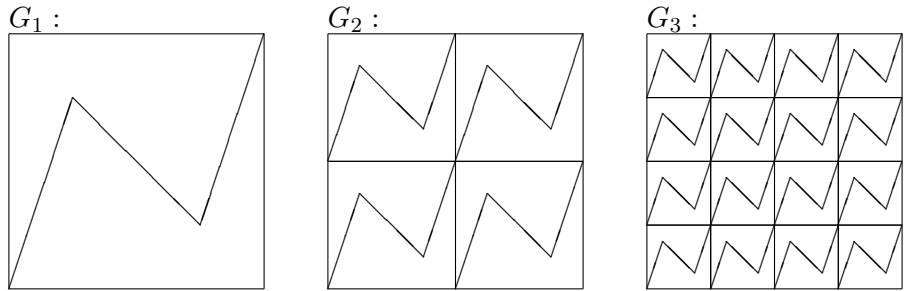
\begin{figure}[H]
\begin{center}\setlength\unitlength{2.4pt}\centering 
 \begin{picture}(140,45)(0,0) \put(0,41){$G_1:$}  \put(50,41){$G_2:$} \put(100,41){$G_3:$} 
  
\def\sq{\begin{picture}(40,40)(0,0) \put(0,0){\line(1,3){10}}  \put(40,40){\line(-1,-3){10}} \put(10,30){\line(1,-1){20}}
  \multiput(0,0)(40,0){2}{\line(0,1){40}}\multiput(0,0)(0,40){2}{\line(1,0){40}} \end{picture} }
  
\put(0,0){\begin{picture}(40,40)(0,0)
  \multiput(0,0)(0,40){1}{\multiput(0,0)(40,0){1}{\sq}} 
  \end{picture} }
  
\put(50,0){\setlength\unitlength{1.2pt}\begin{picture}(40,40)(0,0)
  \multiput(0,0)(0,40){2}{\multiput(0,0)(40,0){2}{\sq}} 
  \end{picture} } 

\put(100,0){\setlength\unitlength{0.6pt}\begin{picture}(40,40)(0,0)
  \multiput(0,0)(0,40){4}{\multiput(0,0)(40,0){4}{\sq}} 
  \end{picture} } 
\end{picture}\end{center}
\caption{The non-convex polygonal grids used in Tables \ref{t1}--\ref{t6}. }
\label{f-5}
\end{figure}

\begin{table}[H]
  \centering  \renewcommand{\arraystretch}{1.1}
  \caption{Error profile by the $P_2$ WG element  for computing \eqref{s2}. }
  \label{t1}
\begin{tabular}{c|cc|cc}
\hline
Grid $G_i$ & \quad $\| u-u_h\|_{0}$ & $O(h^r)$ & \  $\3bar u-u_h\3bar_{1}$& $O(h^r)$   \\ \hline
    &  \multicolumn{4}{c}{On triangular meshes (Figure \ref{f-2}), $\epsilon=10^{-2}$}    \\
\hline  
 3&    0.141E-01 &  2.5&    0.163E+02 &  1.1\\
 4&    0.199E-02 &  2.8&    0.762E+01 &  1.1\\
 5&    0.325E-03 &  2.6&    0.354E+01 &  1.1\\
\hline 
    &  \multicolumn{4}{c}{On triangular meshes (Figure \ref{f-2}), $\epsilon=10^{-7}$}    \\
\hline  
 3&    0.141E-01 &  2.5&    0.163E+02 &  1.1\\
 4&    0.199E-02 &  2.8&    0.762E+01 &  1.1\\
 5&    0.325E-03 &  2.6&    0.354E+01 &  1.1\\
\hline 
    &  \multicolumn{4}{c}{ On  polygonal meshes (Figure \ref{f-5}),  $\epsilon=10^{-2}$}    \\
\hline   
 3&    0.498E-01 &  2.8&    0.307E+02 &  0.9\\
 4&    0.665E-02 &  2.9&    0.154E+02 &  1.0\\
 5&    0.860E-03 &  3.0&    0.775E+01 &  1.0\\
\hline 
    &  \multicolumn{4}{c}{ On  polygonal meshes (Figure \ref{f-5}),  $\epsilon=10^{-7}$}    \\
\hline   
 3&    0.529E-01 &  2.7&    0.307E+02 &  0.8\\
 4&    0.720E-02 &  2.9&    0.155E+02 &  1.0\\
 5&    0.936E-03 &  2.9&    0.777E+01 &  1.0\\
\hline 
    \end{tabular}%
\end{table}%

\begin{table}[H]
  \centering  \renewcommand{\arraystretch}{1.1}
  \caption{Error profile by the $P_3$ WG element  for computing \eqref{s2}. }
  \label{t2}
\begin{tabular}{c|cc|cc}
\hline
Grid $G_i$ & \quad $\| u-u_h\|_{0}$ & $O(h^r)$ & \  $\3bar u-u_h\3bar_{1}$& $O(h^r)$   \\ \hline
    &  \multicolumn{4}{c}{On triangular meshes (Figure \ref{f-2}), $\epsilon=10^{-2}$}    \\
\hline  
 3&    0.740E-03 &  3.8&    0.254E+01 &  2.0\\
 4&    0.537E-04 &  3.8&    0.484E+00 &  2.4\\
 5&    0.988E-05 &  2.4&    0.116E+00 &  2.1\\
\hline 
    &  \multicolumn{4}{c}{On triangular meshes (Figure \ref{f-2}), $\epsilon=10^{-7}$}    \\
\hline  
 3&    0.130E-02 &  3.4&    0.179E+01 &  2.2\\
 4&    0.104E-03 &  3.6&    0.366E+00 &  2.3\\
 5&    0.746E-05 &  3.8&    0.854E-01 &  2.1\\
\hline 
    &  \multicolumn{4}{c}{ On  polygonal meshes (Figure \ref{f-5}),  $\epsilon=10^{-2}$}    \\
\hline   
 3&    0.496E-02 &  3.8&    0.852E+01 &  1.9\\
 4&    0.331E-03 &  3.9&    0.220E+01 &  2.0\\
 5&    0.226E-04 &  3.9&    0.557E+00 &  2.0\\
\hline 
    &  \multicolumn{4}{c}{ On  polygonal meshes (Figure \ref{f-5}),  $\epsilon=10^{-7}$}    \\
\hline   
 3&    0.521E-02 &  3.8&    0.844E+01 &  1.9\\
 4&    0.357E-03 &  3.9&    0.217E+01 &  2.0\\
 5&    0.234E-04 &  3.9&    0.548E+00 &  2.0\\
\hline 
    \end{tabular}%
\end{table}%

\begin{table}[H]
  \centering  \renewcommand{\arraystretch}{1.1}
  \caption{Error profile by the $P_4$ WG element  for computing \eqref{s2}. }
  \label{t3}
\begin{tabular}{c|cc|cc}
\hline
Grid $G_i$ & \quad $\| u-u_h\|_{0}$ & $O(h^r)$ & \  $\3bar u-u_h\3bar_{1}$& $O(h^r)$   \\ \hline
    &  \multicolumn{4}{c}{On triangular meshes (Figure \ref{f-2}), $\epsilon=10^{-2}$}    \\
\hline  
 2&    0.782E-03 &  5.4&    0.125E+01 &  2.8\\
 3&    0.290E-04 &  4.8&    0.134E+00 &  3.2\\
 4&    0.160E-05 &  4.2&    0.161E-01 &  3.1\\
\hline 
    &  \multicolumn{4}{c}{On triangular meshes (Figure \ref{f-2}), $\epsilon=10^{-7}$}    \\
\hline  
 2&    0.864E-03 &  5.3&    0.122E+01 &  2.9\\
 3&    0.292E-04 &  4.9&    0.131E+00 &  3.2\\
 4&    0.105E-05 &  4.8&    0.143E-01 &  3.2\\
\hline 
    &  \multicolumn{4}{c}{ On  polygonal meshes (Figure \ref{f-5}),  $\epsilon=10^{-2}$}    \\
\hline   
 2&    0.113E-01 &  4.8&    0.137E+02 &  2.8\\
 3&    0.366E-03 &  5.0&    0.177E+01 &  3.0\\
 4&    0.115E-04 &  5.0&    0.225E+00 &  3.0\\
\hline 
    &  \multicolumn{4}{c}{ On  polygonal meshes (Figure \ref{f-5}),  $\epsilon=10^{-7}$}    \\
\hline   
 2&    0.116E-01 &  4.8&    0.137E+02 &  2.8\\
 3&    0.376E-03 &  4.9&    0.177E+01 &  3.0\\
 4&    0.121E-04 &  5.0&    0.224E+00 &  3.0\\
\hline 
    \end{tabular}%
\end{table}%

\begin{figure}[H]
 \begin{center}\setlength\unitlength{1.0pt}
\begin{picture}(400,240)(0,0) 
  \put(0,-255){\includegraphics[width=400pt]{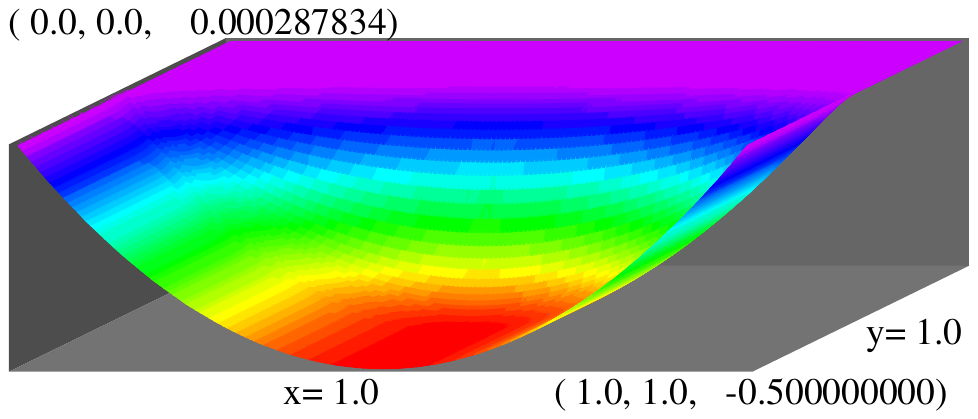}}  
  \put(0,-385){\includegraphics[width=400pt]{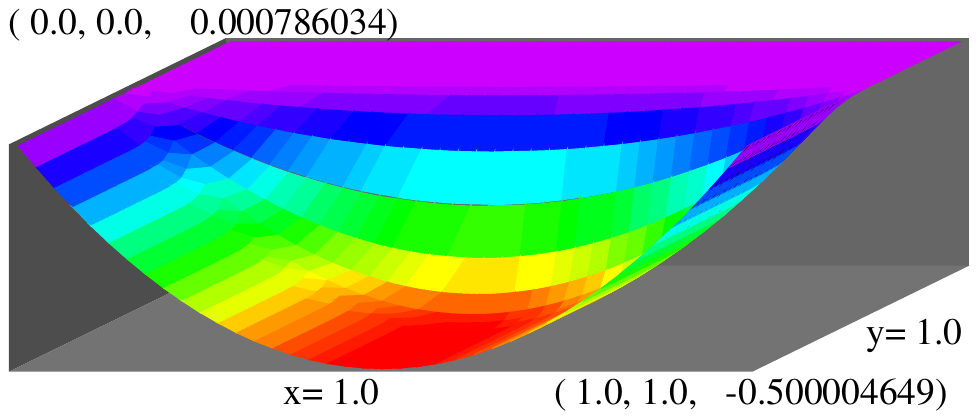}}  
 \end{picture}\end{center}
\caption{The exact solution $u$ in \eqref{s5} (Top) and the $P_4$ WG solution $u_h$ (Bottom) when $\epsilon=10^{-3}$. }\label{f-s5}
\end{figure}

In the second numerical test,  we solve a Cauchy problem with an internal layer.
That is, we solve the Cauchy problem \eqref{model}
   on the unit square domain $\Omega=(0,1)\times(0,1)$, where 
\a{ \Gamma_1=\{0\}\times (0,1)\cup (0,1)\times\{0\}, \ \ \epsilon=10^{-3} \ \t{or} \ 10^{-9 } , \ \ \b b=\p{0\\1}, }
and  $(f, g_1, g_2)$ are chosen so that the exact solution is
\an{\label{s5}
   u =(y^2 - y) (1 + \tanh(20x-10)).  }  
The solution \eqref{s5} and its numerical approximation by $P_4$ WG finite elements are plotted
  in Figure \ref{f-s5}.

 We compute the solution \eqref{s5} on the triangular grids shown in Figure \ref{f-2}, and on the non-convex polygonal 
  grids shown in Figure \ref{f-5}, by 
  the weak Galerkin $P_k$-$P_k$-$P_{k-1}^2$/$P_k^{2\times 2}$ finite elements, $k=2,3$ and $4$.
The results are listed in Tables \ref{t4}-\ref{t6}, 
    where   the   order of convergence is not very stable, due to the ill-posed Cauchy problem and the internal layer.

\begin{table}[H]
  \centering  \renewcommand{\arraystretch}{1.1}
  \caption{Error profile by the $P_2$ WG element  for computing \eqref{s5}. }
  \label{t4}
\begin{tabular}{c|cc|cc}
\hline
Grid $G_i$ & \quad $\| u-u_h\|_{0}$ & $O(h^r)$ & \  $\3bar u-u_h\3bar_{1}$& $O(h^r)$   \\ \hline
    &  \multicolumn{4}{c}{On triangular meshes (Figure \ref{f-2}), $\epsilon=10^{-3}$}    \\
\hline  
 4&    0.321E-02 &  1.1&    0.821E+01 &  1.0\\
 5&    0.560E-03 &  2.5&    0.468E+01 &  0.8\\
 6&    0.138E-03 &  2.0&    0.389E+01 &  0.3\\
\hline 
    &  \multicolumn{4}{c}{On triangular meshes (Figure \ref{f-2}), $\epsilon=10^{-9}$}    \\
\hline  
 4&    0.120E-02 &  2.1&    0.110E+02 &  0.9\\
 5&    0.160E-03 &  2.9&    0.941E+01 &  0.2\\
 6&    0.262E-04 &  2.6&    0.664E+01 &  0.5\\
\hline 
    &  \multicolumn{4}{c}{ On  polygonal meshes (Figure \ref{f-5}),  $\epsilon=10^{-3}$}    \\
\hline   
 4&    0.936E-02 &  1.6&    0.300E+02 &  0.0\\
 5&    0.120E-02 &  3.0&    0.918E+01 &  1.7\\
 6&    0.280E-03 &  2.1&    0.935E+01 &  0.0\\
\hline 
    &  \multicolumn{4}{c}{ On  polygonal meshes (Figure \ref{f-5}),  $\epsilon=10^{-9}$}    \\
\hline   
 4&    0.791E-02 &  1.7&    0.361E+02 &  0.1\\
 5&    0.890E-03 &  3.2&    0.146E+02 &  1.3\\
 6&    0.163E-03 &  2.5&    0.146E+02 &  0.0\\
\hline 
    \end{tabular}%
\end{table}%

\begin{table}[H]
  \centering  \renewcommand{\arraystretch}{1.1}
  \caption{Error profile by the $P_3$ WG element  for computing \eqref{s5}. }
  \label{t5}
\begin{tabular}{c|cc|cc}
\hline
Grid $G_i$ & \quad $\| u-u_h\|_{0}$ & $O(h^r)$ & \  $\3bar u-u_h\3bar_{1}$& $O(h^r)$   \\ \hline
    &  \multicolumn{4}{c}{On triangular meshes (Figure \ref{f-2}), $\epsilon=10^{-3}$}    \\
\hline  
 3&    0.395E-02 &  1.9&    0.922E+01 &  1.8\\
 4&    0.128E-02 &  1.6&    0.737E+01 &  0.3\\
 5&    0.240E-03 &  2.4&    0.410E+01 &  0.8\\
\hline 
    &  \multicolumn{4}{c}{On triangular meshes (Figure \ref{f-2}), $\epsilon=10^{-9}$}    \\
\hline  
 3&    0.318E-02 &  0.8&    0.115E+02 &  1.2\\
 4&    0.456E-03 &  2.8&    0.144E+02 &  0.0\\
 5&    0.482E-04 &  3.2&    0.777E+01 &  0.9\\
\hline 
    &  \multicolumn{4}{c}{ On  polygonal meshes (Figure \ref{f-5}),  $\epsilon=10^{-3}$}    \\
\hline   
 3&    0.348E-01 &  1.5&    0.803E+02 &  0.0\\
 4&    0.511E-02 &  2.8&    0.284E+02 &  1.5\\
 5&    0.104E-02 &  2.3&    0.279E+02 &  0.0\\
\hline 
    &  \multicolumn{4}{c}{ On  polygonal meshes (Figure \ref{f-5}),  $\epsilon=10^{-9}$}    \\
\hline   
 3&    0.396E-01 &  1.1&    0.112E+03 &  0.0\\
 4&    0.384E-02 &  3.4&    0.338E+02 &  1.7\\
 5&    0.634E-03 &  2.6&    0.325E+02 &  0.1\\
\hline 
    \end{tabular}%
\end{table}%

\begin{table}[H]
  \centering  \renewcommand{\arraystretch}{1.1}
  \caption{Error profile by the $P_4$ WG element  for computing \eqref{s5}. }
  \label{t6}
\begin{tabular}{c|cc|cc}
\hline
Grid $G_i$ & \quad $\| u-u_h\|_{0}$ & $O(h^r)$ & \  $\3bar u-u_h\3bar_{1}$& $O(h^r)$   \\ \hline
    &  \multicolumn{4}{c}{On triangular meshes (Figure \ref{f-2}), $\epsilon=10^{-3}$}    \\
\hline  
 2&    0.146E-01 &  3.5&    0.384E+02 &  0.0\\
 3&    0.215E-02 &  2.8&    0.898E+01 &  2.1\\
 4&    0.842E-03 &  1.4&    0.928E+01 &  0.0\\
\hline 
    &  \multicolumn{4}{c}{On triangular meshes (Figure \ref{f-2}), $\epsilon=10^{-9}$}    \\
\hline  
 2&    0.536E-02 &  5.0&    0.224E+02 &  0.2\\
 3&    0.148E-02 &  1.9&    0.160E+02 &  0.5\\
 4&    0.200E-03 &  2.9&    0.153E+02 &  0.1\\
\hline 
    &  \multicolumn{4}{c}{ On  polygonal meshes (Figure \ref{f-5}),  $\epsilon=10^{-3}$}    \\
\hline   
 2&    0.303E+00 &  4.1&    0.340E+03 &  1.8\\
 3&    0.456E-01 &  2.7&    0.147E+03 &  1.2\\
 4&    0.707E-02 &  2.7&    0.144E+03 &  0.0\\
\hline 
    &  \multicolumn{4}{c}{ On  polygonal meshes (Figure \ref{f-5}),  $\epsilon=10^{-9}$}    \\
\hline   
 2&    0.287E+00 &  4.2&    0.320E+03 &  2.0\\
 3&    0.320E-01 &  3.2&    0.161E+03 &  1.0\\
 4&    0.536E-02 &  2.6&    0.150E+03 &  0.1\\
\hline 
    \end{tabular}%
\end{table}%

\end{document}